\documentclass[11pt,reqno,a4paper]{amsart}
 \usepackage{amsfonts}
\usepackage{epsfig}
\usepackage{graphicx}
\usepackage{amsmath}
\usepackage{amssymb}
\usepackage{colordvi}
\usepackage{times,amsmath,epsfig,float,multicol,subfigure}

\usepackage{amsfonts, amssymb, amsmath, amscd, amsthm, txfonts, color, enumerate}

 \usepackage{graphicx}
\usepackage{hyperref}

 \allowdisplaybreaks

\numberwithin{equation}{section}
\numberwithin{figure}{section}
\theoremstyle{plain}
\newtheorem{theorem}{Theorem}[section]

\newtheorem{lemma}[theorem]{Lemma}

\theoremstyle{definition}
\newtheorem{remark}[theorem]{Remark}

\newtheorem{definition}[theorem]{Definition}

\numberwithin{equation}{section}

\title[Density of  the level sets of the metric mean dimension  for homeomorphisms]{Density of  the level sets of the metric mean dimension   for homeomorphisms}

\author{Jeovanny M. Acevedo, Sergio Roma\~na, Raibel Arias}

\address{Jeovanny de Jesus Muentes Acevedo, Facultad de Ciencias B\'asicas,  Universidad Tecnol\'ogica de  Bol\'ivar, Cartagena de Indias, Colombia.}
\email{jmuentes@utb.edu.co}
\address{Sergio Augusto Roma\~na Ibarra,  Instituto de Matemática - Universidade Federal do Rio de Janeiro, CEP 21941-909,
Rio de Janeiro-Brasil and Department of Mathematics, Southern University of Science and Technology, Shenzhen, China. 
}
\email{sergiori@im.ufrj.br}
\address{Raibel de Jesus Arias Cantillo,    Universidade Federal de Maranh\~ao, Campus Balsas, Brasil}
\email{raibel.jac@ufma.br}

\begin{document}

\begin{abstract}
  Let $N$ be an $n$-dimensional compact riemannian manifold, with $n\geq 2$. In this paper, we prove that for any $\alpha\in [0,n]$, the set consisting of homeomorphisms on $N$ with   lower and upper metric mean dimensions equal to $\alpha$ is dense in $\text{Hom}(N)$. More generally, given $\alpha,\beta\in [0,n]$, with $\alpha\leq \beta$, we show the set consisting of homeomorphisms on $N$ with  lower metric mean dimension equal to $\alpha$ and upper  metric mean dimension equal to $\beta$ is dense in $\text{Hom}(N)$. {Furthermore, we also give a proof that the set of homeomorphisms with upper metric mean dimension equal to $n$ is residual in $\text{Hom}(N)$}. 
\end{abstract}


\subjclass[2020]{	 }

\date{\today}
\maketitle

\section{Introduction}

In   the late 1990's, M. Gromov   introduced the notion of mean topological dimension for a continuous map $ \phi: X\rightarrow X$, which is denoted by  $\text{mdim}(X,\phi)$, where $X$ is a compact topological space. The mean topological dimension is an invariant under conjugacy. Furthermore, this is a useful tool in order to characterize dynamical systems that can be embedded in   $(([0,1]^m)^{\mathbb{Z}},\sigma)$, where $\sigma$ is the left shift map on  $([0,1]^m)^{\mathbb{Z}} $ (see \cite{Simbolic}, \cite{Tsukamoto}). In \cite{lind}, Lindenstrauss and Weiss proved  that any homeomorphism $\phi: X\rightarrow X$  that can be  embedded in $(([0,1]^m)^{\mathbb{Z}},\sigma)$ must satisfy that
 $\text{mdim}(X,\phi)\leq m$. In \cite{Gutman},  Gutman and Tsukamoto showed that,   if $(X, \phi)$ is a minimal system with $\text{mdim}(X, \phi) <m/
2$, then we can embed it in $(([0, 1]^{m})^{\mathbb{Z}},\sigma)$. In \cite{lind4}, Lindenstrauss and Tsukamoto  presented an example of a minimal system with mean topological dimension equal to $m/2$ that cannot be
embedded into $(([0, 1]^{m})^{\mathbb{Z}},\sigma)$, which show  the constant $m/2$  is optimal. Some applications in information theory can be found in  \cite{lind3} and \cite{Lindestr}.

\medskip

The mean topological   dimension is difficult to calculate. Therefore,  Lindenstrauss and Weiss   in \cite{lind} introduced the notion of   metric mean dimension, which is an  upper bound for the mean topological  dimension. The metric mean dimension is a  metric-dependent quantity (this dependence is not continuous, as we can see in \cite{Aceyot}), therefore, it is not an invariant under topological conjugacy. 

\subsection{Metric mean dimension}\label{section1}
   
  Let $X$ be a compact metric space  endowed with a metric $d$. For any  
$n\in\mathbb{N}$, we define $d_n:X\times X\to [0,\infty)$ by
$$
d_n(x,y)=\max\{d(x,y),d(\phi(x),\phi(y)),\dots,d(\phi^{n-1}(x),\phi^{n-1}(y))\}.
$$   Fix $\varepsilon>0$.  We say that $A\subset X$ is an $(n,\phi,\varepsilon)$-\textit{separated} set
if $d_n(x,y)>\varepsilon$, for any two  distinct points  $x,y\in A$. We denote by $\text{sep}(n,\phi,\varepsilon)$ the maximal cardinality of any  $(n,\phi,\varepsilon)$-{separated}
subset of $X$. Set $$\text{sep}(\phi,\varepsilon)=\underset{n\to\infty}\limsup \frac{1}{n}\log \text{sep}(n,\phi,\varepsilon).$$

 We say that $E\subset X$ is an $(n,\phi,\varepsilon)$-\textit{spanning} set for $X$ if
for any $x\in X$ there exists $y\in E$ such  that $d_n(x,y)<\varepsilon$. Let $\text{span}(n,\phi,\varepsilon)$ be the minimum cardinality
of any $(n,\phi,\varepsilon)$-spanning subset of $X$.    Set $$\text{span}(\phi,\varepsilon)=\underset{n\to\infty}\limsup \frac{1}{n}\log \text{span}(n,\phi,\varepsilon).$$

 \begin{definition}
  The \emph{topological entropy} of $\phi:X\rightarrow X$   is defined by      
  \begin{equation*}\label{topent}h_{\text{top}}(\phi)=\lim _{\varepsilon\to0} \text{sep}(\phi,\varepsilon)=\lim_{\varepsilon\to0}  \text{span}(\phi,\varepsilon).
\end{equation*}
  \end{definition}

 \begin{definition}
  We define the \emph{lower  metric mean dimension}   and the \emph{upper metric mean dimension} of $(X,d,\phi )$ by
  \begin{align*}\label{metric-mean}
 \underline{\text{mdim}}_{\text{M}}(X,d,\phi)&=\liminf_{\varepsilon\to0} \frac{\text{sep}(\phi,\varepsilon)}{|\log \varepsilon|}= \liminf_{\varepsilon\to0} \frac{\text{span}(\phi,\varepsilon)}{|\log \varepsilon|}\\
 \text{and }\quad\overline{\text{mdim}}_{\text{M}}(X,d,\phi)&=\limsup_{\varepsilon\to0} \frac{\text{sep}(\phi,\varepsilon)}{|\log \varepsilon|}=\limsup_{\varepsilon\to0} \frac{\text{span}(\phi,\varepsilon)}{|\log \varepsilon|},
\end{align*}
respectively. \end{definition}
 
\begin{remark}\label{erfer}
 Throughout the paper, we will omit the underline and the overline  on the notations $\underline{\text{mdim}}_{\text{M}}$ and $\overline{\text{mdim}}_{\text{M}}$      when the result be valid for both cases, that is,  we  will use $\text{mdim}_{\text{M}}$ for the both cases. 
 \end{remark}

In recent years, the metric mean dimension has been the subject of multiple investigations, which can be verified in the bibliography of the present work. The purpose of this manuscript is to complete the research  started in   \cite{Carvalho},  \cite{VV} and \cite{JeoPMD}, concerning to the topological properties of the level sets   of the  metric mean dimension map. 
 
 \medskip 
 
In \cite{Carvalho}, Theorem C, the authors proved the set consisting of continuous maps $\phi:[0,1]\rightarrow [0,1]$ such that $\underline{\text{mdim}}_\text{M}([0,1],  |\cdot|,\phi)=\overline{\text{mdim}}_\text{M}([0,1],  |\cdot|,\phi)=\alpha$, for a fixed $\alpha\in[0,1]$, is dense in $C^{0}([0,1])$ (see also  \cite{JeoPMD}, Theorem 4.1). Furthermore, in Theorem A they showed if $N$ is an $n$-dimensional  compact riemannian manifold, with $n\geq 2$, and  riemannian metric $d$,    the set  of homeomorphisms $\phi:N\rightarrow N$ such that $\overline{\text{mdim}}_\text{M}(N,  d,\phi)=n$ contains a residual set in $\text{Hom}(N)$ (see a particular case of this fact in \cite{VV},  Proposition 10).  Next, for any $n\geq 1$,   the set consisting of continuous maps $\phi:N\rightarrow N$ such that $$\underline{\text{mdim}}_\text{M}(N,  d,\phi)=\overline{\text{mdim}}_\text{M}(N,  d,\phi)=\alpha$$  is dense  in $C^{0}(N)$, for a fixed   $\alpha\in [0,n]$ (see \cite{JeoPMD}, Theorem 4.5, and \cite{JeoPMD2}, Theorem 3.6). 
\medskip

We consider the next level sets of the metric mean dimension for homeomorphisms:

  \begin{definition}
 For $\alpha,\beta\in [0,n]$, with $\alpha \leq \beta$, we will set $$ {H}_{\alpha}^{\beta}(N)=\{\phi\in \text{Hom}(N):  \underline{\text{mdim}}_\text{M}(N,  d,\phi)=\alpha \, \, \text{and} \, \, \overline{\text{mdim}}_\text{M}(N,  d,\phi)=\beta\}.$$
If $\alpha=\beta$, we denote ${H}_{\alpha}^{\alpha}(N)$ by ${H}_{\alpha}(N)$.
 \end{definition}

If $n=1$, then ${H}_{\alpha}^{\beta}(N)=\emptyset$ for $0<\alpha\leq\beta\leq 1$. This is due to the fact that any homeomorphism on a one-dimensional compact Riemannian manifold has zero topological entropy, leading to zero metric mean dimension. Our initial result is presented in the following theorem, proved in \cite{JeoPMD2}, Theorem 3.6, specifically for continuous maps on the interval.  

\begin{theorem}\label{density_tipo1} Let $n\geq 2$. 
 For any  $\alpha,\beta\in [0,n]$, with $\alpha\leq \beta$, the set  ${H}_{\alpha}^{\beta}(N)$ is dense in $\emph{Hom}(N)$.
\end{theorem}
    Using the techniques employed to prove Theorem \ref{density_tipo1}, we will provide a new   proof of Theorem A in \cite{Carvalho}, that is:  
\begin{theorem}\label{teoresidual*}
The set $\overline{H}(N)=\{\phi\in \emph{Hom}(N):\overline{\emph{mdim}}_{\emph{M}}(N,d,\phi)=n\}$ contains a residual set in $\emph{Hom}(N)$. 
\end{theorem}

Yano, in \cite{Yano},  defined a  kind of horseshoe in order to prove the set consisting of homeomorphisms $\phi:N\rightarrow N$   with infinite entropy is generic in $\text{Hom}(N)$, where $N$ is an $n$ dimensional compact manifold, with $n\geq 2$.  
If we want to construct a continuous map  with infinite entropy we can consider an infinite sequence of horseshoes such that the number of legs   is unbounded. For the metric mean dimension case, in \cite{Carvalho} and \cite{VV} the authors used horseshoe in order to prove the set consisting of homeomorphisms $\phi:N\rightarrow N$   with upper metric mean dimension equal to $n$ (which is the maximal value of the metric mean dimension for any map defined on $N$)  is generic in $\text{Hom}(N)$. To get metric mean dimension equal to $n$ we must construct a sequence of horseshoes such that the number of legs increases very quickly compared to the   decrease in their diameters.

\medskip 

Estimating the precise value of the metric mean dimension for a homeomorphism, and hence to obtain a homeomorphism with  metric mean dimension equal to  a fixed $\alpha\in(0,n)$, is harder and not trivial sake. We need to establish a precise  relation between  the sizes of the horseshoes together with the number of appropriated legs to control the metric mean dimension. This is our main tool (see Lemma \ref{lemma22}). 

\medskip 

 The paper is organized as follows: In Section \ref{section55}, we will construct  homeomorphisms, defined  on an $n$-cube, with metric mean dimension equal to $\alpha$, for a fixed $\alpha\in (0,n]$. Furthermore, 
given $\alpha,\beta\in[0,n]$, with $\alpha\leq \beta$, we will construct  examples of homeomorphisms $\phi:[0,1]^{n}\rightarrow [0,1]^{n}$ such that $       \underline{\text{mdim}}_{\text{M}}([0,1]^{n} ,d,\phi)=\alpha$ and $  \overline{\text{mdim}}_{\text{M}}([0,1]^{n} ,d,\phi)=\beta.$     In Section \ref{secvvnfnrk}, we will prove   Theorem \ref{density_tipo1}. Finally, in Section \ref{Section6}, we will show  Theorem \ref{teoresidual*}.

  \section{Homeomorphisms on  the {n}-cube with positive metric mean dimension}\label{section55}

Let $n\geq 2$. Given $\alpha, \beta \in [0,n]$, $\alpha\leq \beta$, in this section we will construct  a homeomorphism $\phi_{\alpha, \beta}$, defined on an $n$-cube, with lower metric mean dimension equal to $\alpha$ and upper metric mean dimension equal to $\beta$ (see Lemma \ref{med1}). This construction will be the main tool to prove the Theorem \ref{density_tipo1}, since if a homeomorphism present a periodic orbit, then we can  glue, in the  $C^0$-topology, along this orbit the dynamic of $\phi_{\alpha,\beta}$. 

\medskip

{The construction of $\phi_{\alpha, \beta}$ requires a special type of horseshoe. So, let us present  the first definition.}

 \begin{definition}[$n$-dimensional   $(2k+1)^{n-1}$-horseshoe]\label{3dimensionalhorse}  
 Fix $n\geq 2$. Take $E=[a,b]^{n}$ and set $|E|=b-a$. For a fixed natural number $k>1$,   take the sequence    $a=t_0<t_{1}<\cdots<t_{4k}<t_{4k+1}=b $, with $|t_{i}-t_{i-1}|=\frac{b-a}{4k+1}$,
 and  consider   $$H_{i_{1},i_{2},\dots,i_{n-1}}=[a,b]\times[t_{i_{1}-1},t_{i_{1}}]\times  \cdots \times[t_{i_{n-1}-1}, t_{i_{n-1}}],\quad \text{for }i_{j}\in\{1,\dots, 4k+1\}.$$   Take $a=s_0<s_{1}<\cdots<s_{2(2k+1)^{n-1}-2}<s_{2(2k+1)^{n-1}-1}=b $, with $|s_{i}-s_{i-1}|=\frac{b-a}{2(2k+1)^{n-1}-1}$, 
 and consider     $$ V_{l}=[s_{l-1}, s_{l}]\times [a,b]^{n-1},\quad \text{for } l=1, 2, \dots, 2(2k+1)^{n-1}-1.$$ 
  We say that $E\subseteq A\subseteq \mathbb{R}^{n}$ is an     $n$-\textit{dimensional       $(2k+1)^{n-1}$-horseshoe} for a homeomorphism $\phi:A\rightarrow A$ if:   
\begin{itemize} 
\item $\phi(a,a,\dots,a,b)=(a,a,\dots,a,b)$ and $\phi(b,b,\dots,b,a)=(b,b,\dots,b,a)$;
\item For  any $H_{i_{1},i_{2},\dots,i_{n-1}}$, with $i_{j}\in\{1,3,\dots, 4k+1\}$, there exists  some   $l\in\{1,3,\dots, 2(2k+1)^{n-1}-1\}$ with $$\phi(V_{l})=H_{i_{1},i_{2},\dots,i_{n-1}}  \quad\text{and}\quad \phi|_{V_{l}}:V_{l}\rightarrow H_{i_{1},i_{2},\dots,i_{n-1}} \quad \text{is linear}.$$ 
\item For any $l=2,4,\dots, 2(2k+1)^{n-1}-2$,  $\phi(V_{l})\subseteq A\setminus E$.\end{itemize} 
In that case, the sets $H_{i_{1},i_{2},\dots,i_{n-1}}$, with $i_{j}\in\{1,3,\dots, 4k+1\}$ are called the \textit{legs}
 of the horseshoe.   \end{definition}  \begin{figure}[hbtp]
 \centering
 \subfigure[$V_{i}$, $i=1,\dots,9$]{\includegraphics[scale=.28]{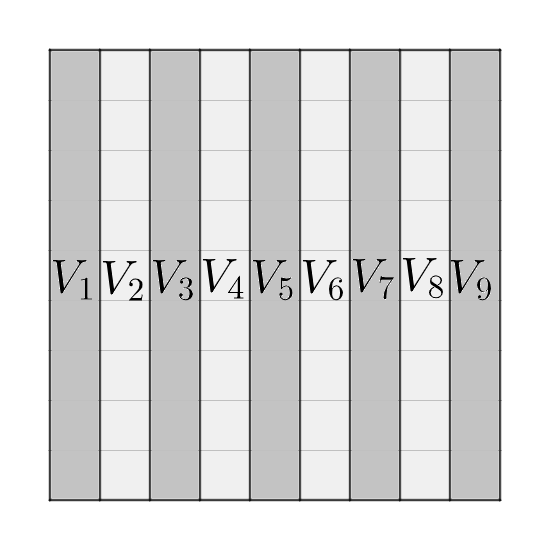}}\quad  \subfigure[$H_{i}$, $i=1,3,\dots,9$]{\includegraphics[scale=.28]{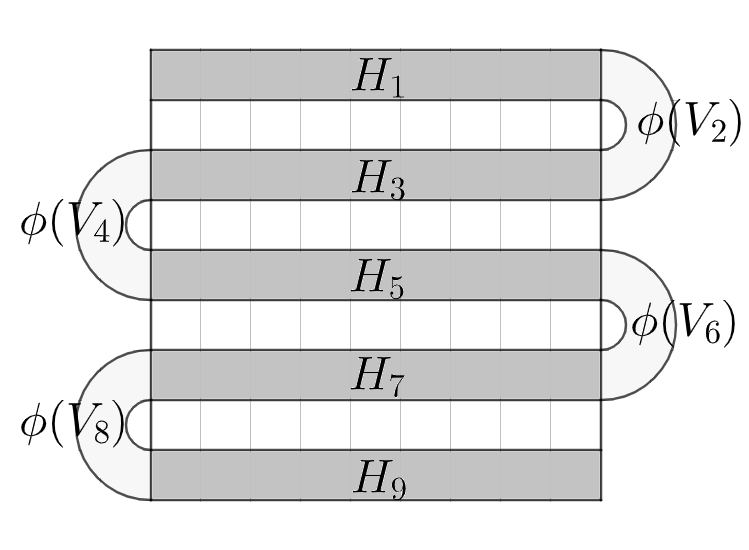}}
 \caption{2-dimensional $5$-horseshoe}\label{l}
 \end{figure}
   
\begin{figure}[hbtp]
 \centering
  \subfigure[$V_{i}$, $i=1,\dots,17$]{\includegraphics[scale=.28]{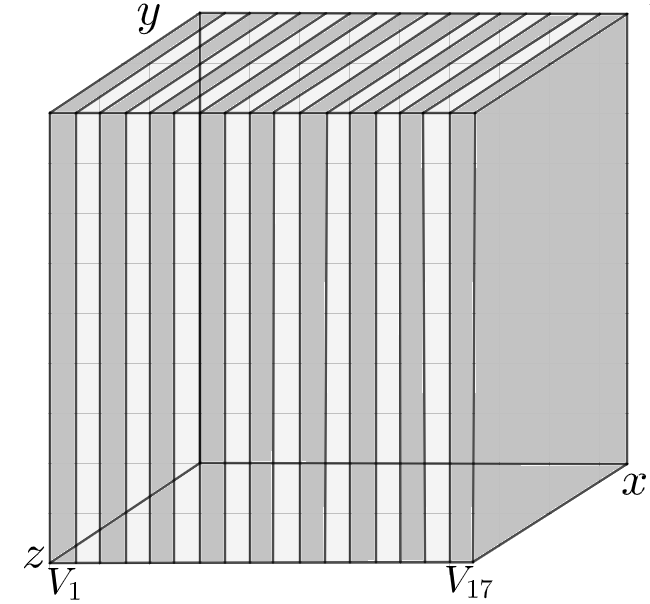}}\quad
 \subfigure[$H_{i,j}$, $i,j\in\{1,3,5\}$]{\includegraphics[scale=.28]{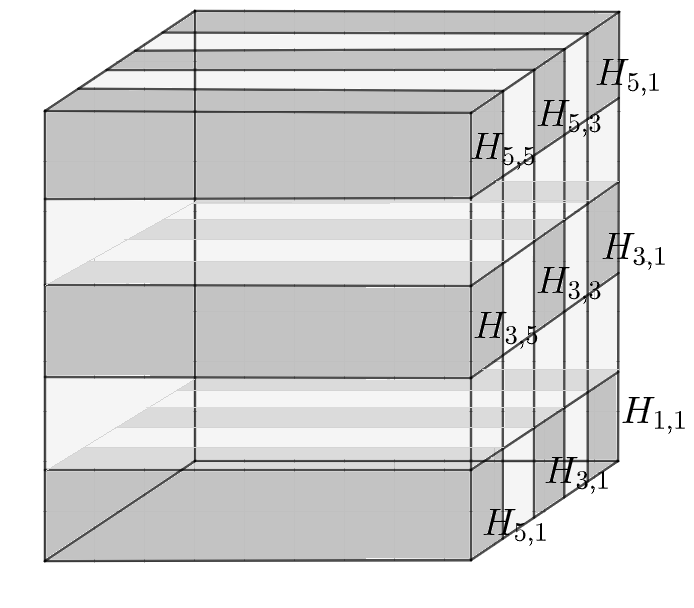} }
 \caption{3-dimensional  9-horseshoe}\label{mm}
 \end{figure}

 In Figure \ref{l} we present an example of a 2-dimensional 5-horseshoe. In that case, $k=2$, we have $2(2k+1)-1$ divisions  $V_{i}$ of $[a,b]^{2}$, and $2k+1$ legs $H_{i}$, for $i=1,3,5,7,9$. In Figure \ref{mm} we have a 3-dimensional 9-horseshoe (we only  show $\phi(E)\cap E$ in that figure). In that case, $k=1$, we have $2(2k+1)^{2}-1$ divisions  $V_{i}$ of $[a,b]^{3}$, and $(2k+1)^{2}$ legs $H_{i,j}$, for $i,j=1,3,5$.
  
  \medskip 
  
  Note if $E=[a,b]^{n}$ is an $n$-dimensional $(2k+1)^{n-1}$-horseshoe for $\phi$, then $$\phi^{2}(V_{l_{1}}\cap H_{i_{1}^{(1)},\dots,i_{n-1}^{(1)}})\cap V_{l_{2}}\cap H_{i_{1}^{(2)},\dots,i_{n-1}^{(2)}}\neq \emptyset ,$$ for any $l_{1},l_{2}\in \{1,3\dots, 2(2k+1)^{n-1}\}$ and $i_{1}^{(1)},\dots,i_{n-1}^{(1)},i_{1}^{(2)},\dots,i_{n-1}^{(2)}\in \{1,3,\dots, 4k+1\}$ (see Figure \ref{lw}). 
  \begin{figure}[hbtp]
 \centering
 \subfigure[$V_{i}$,  $H_{i}$]{\includegraphics[scale=.22]{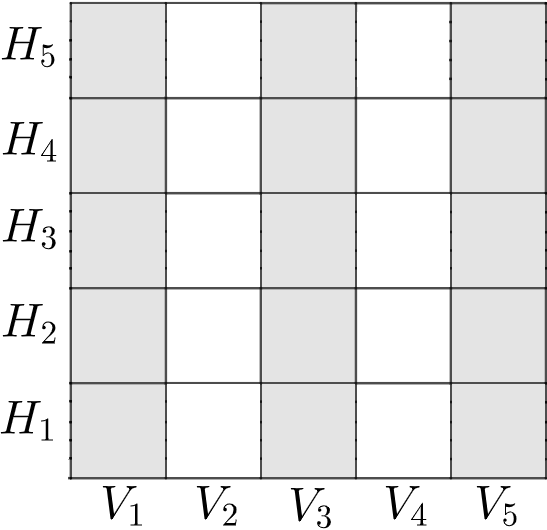}}   \subfigure[$\phi(E)$]{\includegraphics[scale=.24]{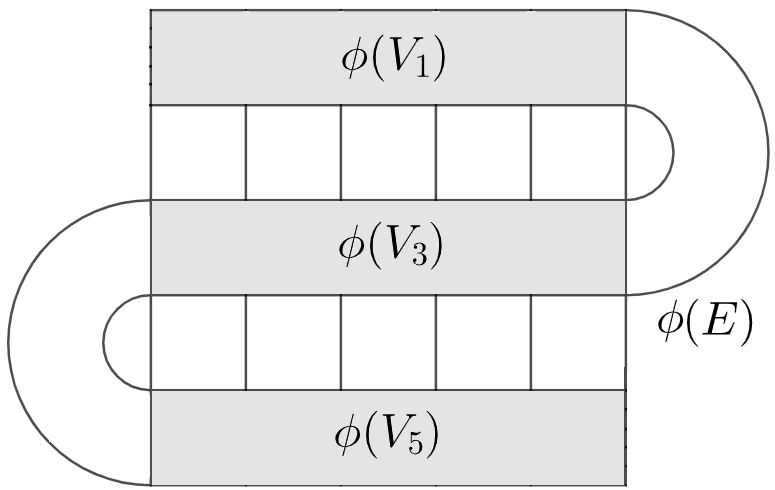}}   \subfigure[$\phi^{2}(E)$]{\includegraphics[scale=.24]{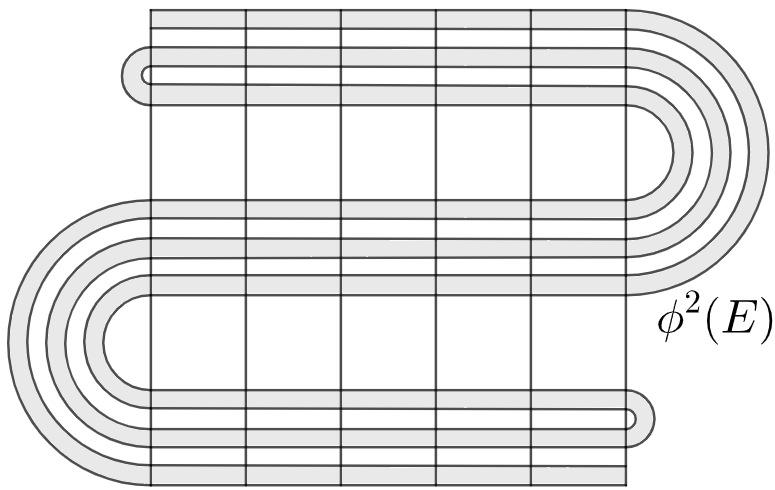}}
 \caption{$E$ is the first square, is a 2-dimensional 3-horseshoe for $\phi$}\label{lw}
 \end{figure}
  
  \medskip

  The assumptions $\phi(a,a,\dots,a,b)=(a,a,\dots,a,b)$ and $\phi(b,b,\dots,b,a)=(b,b,\dots,b,a)$   is just to be able to extend $\phi$ to a homeomorphism $\tilde{\phi}:E^{\prime}\rightarrow E^{\prime}$, where $E^{\prime}$ is an $n$-cube with $E\subset E^{\prime}$, such that $\tilde{\phi}|_{\partial E^{\prime}}\equiv Id$ and $h_{\text{top}}(\phi^{\prime})=h_{\text{top}}(\phi|_{E\cap \phi(E )})$ (see Lemma \ref{lemma2}). Thus, our strategy to prove the Theorem \ref{density_tipo1} will be to use local charts to glue such horseshoes along of periodic orbits of a homeomorphism.

\begin{lemma}\label{lemma2} Let $E, E^{\prime}\subseteq \mathbb{R}^n$ be closed $n$-cubes  with $E \subsetneqq (E^{\prime})^{\circ}$ and fix $k\in\mathbb{N}$. There   exists a homeomorphism $\phi: E^{\prime}\longrightarrow E^{\prime}$ such that $\phi|_{E}:  E \longrightarrow E$ is an $n$-dimensional $(2k+1)$-horseshoe, $\phi|_{\partial E^{\prime}}\equiv Id$ and $h_{\emph{top}}(\phi)=h_{\emph{top}}(\phi|_{E\cap \phi(E )})$.\end{lemma}

     Inspired by the results shown in \cite{Carvalho} and \cite{VV}     to obtain a homeomorphism  $\phi:N\rightarrow N$ with upper metric mean dimension equal to $\text{dim}(N)$, 
   we present the next  lemma,  which proves for any $\alpha\in (0,n]$, there exists a homeomorphism  $\phi:\mathcal{\mathcal{C}} := [0,1]^n\rightarrow \mathcal{\mathcal{C}}$, with (upper and lower) metric mean dimension equal to $\alpha$.  
   
   \medskip 
   
   On any  {$n$-cube} $E\subseteq\mathbb{R}^{n}$, we will consider the metric inherited from $\mathbb{R}^{n}$, $\Vert \cdot\Vert$, given by $$\Vert (x_{1},\dots,x_{n})-(y_{1},\dots,y_{n})\Vert= \sqrt{(x_{1}-y_{1})^{2}+\cdots+(x_{n}-y_{n})^{2}}.$$

\begin{lemma}\label{lemma22}  Let $\phi:  {\mathcal{C}}\rightarrow  {\mathcal{C}}$ be a homeomorphism,  $E_{k}=[a_k , b_k]^{n}$ and  $E_{k}^{\prime}$   sequences of cubes such that:
\begin{itemize}\item[C1.] $ E_{k}\subset (E_{k}^{\prime})^{\circ}$ and $(E_{k}^{\prime})^{\circ}\cap (E_{s}^{\prime})^{\circ}=\emptyset$ for $k\neq s$.
\item[C2.]   $S:=\cup_{k=1}^{\infty}E^{\prime}_{k}\subseteq {\mathcal{C}}$. 
\item[C3.] each $E_{k}$ is an $n$-dimensional   $3^{k(n-1)}$-horseshoe for $\phi$;     
\item[C4.] For each $k$, $\phi|_{E_{k}^{\prime}}: E_{k}^{\prime}\rightarrow E_{k}^{\prime}$ satisfies the properties in Lemma \ref{lemma2};
\item[C5.] $\phi|_{{\mathcal{C}}\setminus S}: {\mathcal{C}}\setminus S\rightarrow  {\mathcal{C}}\setminus S$ is the identity. 
\end{itemize}
 We have:
 \begin{itemize}\item[(i)]  For a fixed $r\in (0,\infty)$, if $|E_{k}|=\frac{B}{3^{kr}}$ for each $k\in\mathbb{N}$, where $B>0$ is a constant, then $$  \emph{mdim}_{\emph{M}}({\mathcal{C}},\Vert \cdot\Vert,\phi^{2}) =\frac{n}{r+1} . $$
 \item[(ii)] If $|E_{k}|=\frac{B}{k^{2}}$ for each $k\in\mathbb{N}$, where $B>0$ is a constant, then $$  \emph{mdim}_{\emph{M}}( {\mathcal{C}},\Vert \cdot\Vert,\phi^{2}) =n . $$
 \end{itemize}\end{lemma}
\begin{proof} We will prove (i), since (ii) can be proved  analogously. Set $\varphi=\phi^{2}$. Note that 
$$  {{\text{mdim}}_{\text{M}}}( {\mathcal{C}},\Vert \cdot\Vert,\varphi ) ={{\text{mdim}}_{\text{M}}}(S,\Vert \cdot\Vert,\varphi|_{S} ). $$

Take any $\varepsilon\in (0,1)$. For any $k\geq 1$, set $ \varepsilon_k=  \frac{|E_{k}|}{2(3^{k})-1}  =\frac{B}{(2(3^{k})-1)3^{kr}} $. There exists $k\geq 1$ such that  $\varepsilon \in [\varepsilon_{k+1}, \varepsilon_{k}]$. We have  
 \begin{equation*}   \text{sep}( n, \varphi   , \varepsilon)\geq \text{sep}( n, \varphi   , \varepsilon_{k})\geq   \text{sep}( n, \varphi  |_{E_{k}\cap \phi(E_{k})}, \varepsilon_{k}) \quad\text{for any }n\geq 1 .  \end{equation*}
Since $E_{k}$ is an $n$-dimensional $3^{k(n-1)}$-horseshoe for $\phi$, for each $k\geq 1$, consider   $${H}^{k}_{i_{1},i_{2},\dots,i_{n-1}}=[a_{k},b_{k}]\times[t_{i_{1}-1},t_{i_{1}}]\times  \cdots \times[t_{i_{n-1}-1}, t_{i_{n-1}}],\quad\text{for }i_{j}=1,\dots, 2(3^{k}) -1,$$   and  $$ V_{l}^{k}=[s_{l-1}, s_{l}]\times[a_{k},b_{k}]^{n-1},\quad\text{for }l=1,\dots, 2(3^{k})^{n-1}-1,$$ as in Definition \ref{3dimensionalhorse}.  
 For each $j=1,\dots, n-1$, let $i_{j}\in\{1,3,5,\dots,  2(3^{k})-1\}$  
 and  take $$C^{k}_{l,i_{1},\dots,i_{n-1}} = H^{k}_{i_{1},i_{2},\dots,i_{n-1}}\cap V^{k}_{l} \quad\text{for each }l\in \{1,3,\dots, 2(3^{k})^{n-1}-1\}.$$ 
For $t=1,\dots,m$, let  $i_{1}^{(t)},$ $ \dots,i_{n-1}^{(t)}\in\{1,3,\dots, 2(3^{k}) -1\}$, $l_{j_{t}}\in \{1,3,\dots, 2(3^{k})^{n-1}-1\}$  and 
set   \begin{align*} 
 C^{k}_{l_{j_{2}}, i_{1}^{(2)},\dots, i_{n-1}^{(2)},l_{j_{1}},i_{1}^{(1)},\dots, i_{n-1}^{(1)}} &=\varphi^{-1}\left[\varphi\left(C^{k}_{ l_{j_{2}},i_{1}^{(2)},\dots, i_{n-1}^{(2)}}\right)\cap C^{k}_{l_{j_{1}},i_{1}^{(1)},\dots,i_{n-1}^{(1)}}\right]\\
  &\vdots\\
 C^{k}_{l_{j_{m}}, i_{1}^{(m)},\dots,i_{n-1}^{(m)},  \dots, l_{j_{1}},i_{1}^{(1)},  i_{1}^{(1)},\dots, i_{n-1}^{(1)}} &=\varphi^{-(m-1)}\left[\varphi^{m-1}\left(C^{k}_{l_{j_{m}},i_{1}^{(m)},\dots,i_{n-1}^{(m)},  \dots,l_{j_{2}}, i_{1}^{(2)},\dots,i_{n-1}^{(2)}}\right)\cap C^{k}_{l_{j_{1}},i_{1}^{(1)},\dots,i_{n-1}^{(1)}}\right]
 \end{align*} 
 From the definition of $\varphi$, these sets are non-empty. Furthermore, we can choose a sequence $s_{1},\dots,s_{3^{k}}\in \{1,3,\dots, 2(3^{k})^{n-1}-1\}$ such that,  if $x $ and $y$ belong to different sets  $ C^{k}_{s^{(m)},i_{1}^{(m)},\dots,i_{n-1}^{(m)},  \dots,s_{(m)}, i_{1}^{(1)},\dots, i_{n-1}^{(1)}}  $, where $s^{(j)}\in \{s_{1},\dots,s_{3^{k}}\}$, 
 then $d_{m}(x,y)>\varepsilon_{k}$. Note that  we have   $3^{km}3^{k(n-1) m}= 3^{knm}$ sets of this form. Hence, 
 \begin{equation*}    \text{sep}(m, \varphi   , \varepsilon_{k}) \geq     3^{knm} \quad \text{and thus}\quad 
  \text{sep}(\varphi  , \varepsilon)\geq    \log  {3^{kn}}.
 \end{equation*}
 Therefore,  \begin{align*}\label{exxample12} {\underline{\text{mdim}}_{\text{M}}}(S ,\Vert \cdot\Vert,\varphi|_{S} )& = \liminf_{\varepsilon\rightarrow 0} \frac{\text{sep}(\varphi  ,\varepsilon )}{|\log  {\varepsilon}|} \geq \lim_{k\rightarrow \infty}\frac{\log 3^{kn}}{|\log \varepsilon_{k+1}|} \\
 &=\lim_{k\rightarrow \infty}\frac{\log 3^{kn}}{\log [(2(3^{k+1})-1)3^{(k+1)r}B^{-1}]} = \lim_{k\rightarrow \infty}\frac{\log 3^{kn}}{\log [3^{(k+1)+(k+1)r} ]} \\
 &=\frac{n}{r+1}.\end{align*}

On the other hand, note that $ \frac{\log 3^{kn}}{\log[4(2(3^{k})-1)3^{kr}B^{-1}]} \rightarrow\frac{n}{1+r}$ as $k\rightarrow \infty$. Hence, for any $\delta >0$, there exists $k_{0}\geq 1$, such that,  for any  $k> k_{0}$, we have $ \frac{\log 3^{kn}}{\log[4(2(3^{k})-1)3^{kr}B^{-1}]} <\frac{n}{1+r}+\delta$. Hence, suppose that $\varepsilon>0$ is small enough such that $\varepsilon<\epsilon_{k_{0}}$.   Set $ \tilde{\Omega}_{k}=\underset{n\in\mathbb{Z}}{\bigcap}\varphi^{n}(E_{k})$ and take $\tilde{\Omega}=\underset{k\in\mathbb{N}}{\bigcup}\Omega_{k}$.    We have $$  {{\text{mdim}}_{\text{M}}}( {\mathcal{C}},\Vert \cdot\Vert,\varphi ) ={{\text{mdim}}_{\text{M}}}(\tilde{\Omega},\Vert \cdot\Vert,\varphi|_{\tilde{\Omega}} ). $$ 
If $x\in \tilde{\Omega}$, then $x$ belongs to some $C^{k}_{l_{j_{m}},i_{1}^{(m)},\dots,i_{n-1}^{(m)},  \dots,l_{j_{1}}, i_{1}^{(1)},\dots, i_{n-1}^{(1)}}$. Take   $s^{(j_{1})},\dots s^{(j_{m})} \in \{s_{1},\dots, s_{3^{k}}\}$ with $$|s^{(j_{i})}-l_{j_{i}}|=\underset{t=1,\dots,m}{\min}|s^{(t)}-l_{j_{i}}|,\quad \text{fot each }i=1,\dots,m,$$ and  furthermore $d_{m}(x,y)\leq 4 \varepsilon_{k}$ for any $y\in C^{k}_{s^{(j_{m})},i_{1}^{(m)},\dots,i_{n-1}^{(m)},  \dots,s^{(j_{1})}, i_{1}^{(1)},\dots, i_{n-1}^{(1)}}$. Hence, if $Y_{k}=\cup_{j=1}^{k}E_{j}$, for every $ m\geq1$, we have 
  \begin{align*}    \text{span}(m,\varphi  |_{Y_{k}}, 4\varepsilon) & \leq \sum_{j=1}^{k}   \frac{3^{jnm}}{\varepsilon}  \leq k\frac{3^{knm}}{\varepsilon} .  \end{align*} 
 Therefore,  
 \begin{align*} \frac{\text{span}(\varphi  |_{Y_{k}}  ,4 \varepsilon)}{|\log 4\varepsilon|}& \leq\limsup_{m\rightarrow \infty} \frac{\log\left[  k\frac{3^{knm}}{\varepsilon} \right]}{m |\log  4\varepsilon_{k}|}  = \frac{\log 3^{kn}}{\log[4(2(3^{k})-1)3^{kr}B^{-1}]} <\frac{n}{1+r}+\delta.
 \end{align*}
This fact implies that for any $\delta >0$ we have \begin{align*}  {\overline{\text{mdim}}_{\text{M}}}(S ,\Vert \cdot\Vert,\varphi|_{S} )<\frac{n}{r+1}+\delta\quad\text{and hence }\quad {\overline{\text{mdim}}_{\text{M}}}(S ,\Vert \cdot\Vert,\varphi|_{S})\leq\frac{n}{r+1}.\end{align*}
  The above facts prove that ${ {\text{mdim}}_{\text{M}}}(\mathcal{C} ,\Vert \cdot\Vert,\varphi)=\frac{n}{r+1}. $   
\end{proof}

Given $\alpha,\beta\in[0,n]$, with $\alpha\leq \beta$,  {in the next lemma we construct an homeomorphism} $\phi:[0,1]^{n}\rightarrow [0,1]^{n}$ such that $$       \underline{\text{mdim}}_{\text{M}}([0,1]^{n} ,\Vert \cdot\Vert,\phi)=\alpha\quad \text{and}\quad   \overline{\text{mdim}}_{\text{M}}([0,1]^{n} ,\Vert \cdot\Vert,\phi)=\beta.$$

\begin{lemma}\label{med1}
For every $\alpha,\beta \in [0,n]$, with $\alpha<\beta$,   there exists  a homeomorphism $\phi_{\alpha,\beta}:[0,1]^{n}\rightarrow [0,1]^{n}$ such that $$ \alpha=     \underline{\emph{mdim}}_{\emph{M}}([0,1]^{n} ,\Vert \cdot\Vert,\phi_{\alpha,\beta})<  \overline{\emph{mdim}}_{\emph{M}}([0,1]^{n} ,\Vert \cdot\Vert,\phi_{\alpha,\beta})=\beta.$$ 
\end{lemma}
 
\begin{proof}  
 We will prove the case $\beta<n$.   First, we will show there exists a homeomorphism $\phi_{0,\beta}:[0,1]^{n}\rightarrow [0,1]^{n}$  such that $$ 0=     \underline{\text{mdim}}_{\text{M}}([0,1]^{n} ,\Vert \cdot\Vert,\phi_{0,\beta})<  \overline{\text{mdim}}_{\text{M}}([0,1]^{n} ,\Vert \cdot\Vert,\phi_{0,\beta})=\beta.$$ 
 Fix  $r>0$ and set $ b=\frac{n}{r+1}$.  Set $a_{0}=0$ and $a_{n}= \sum_{i=0}^{n-1}\frac{C}{3^{ir}}$ for $n\geq 1$, where $C=\frac{1}{\sum_{i=0}^{\infty}\frac{1}{3^{ir}}}= \frac{3^{r}-1}{3^{r}}$. Let $E_{k}=[a_{k},b_{k}]^{n}$ and $E_{k}^{\prime}$ satisfying the conditions C1 and C2, C4 and C5 in     Lemma \ref{lemma22}, with $|E_{k}|=\frac{B}{3^{kr}}$, and such that each $E_{k^{k}}$ is an $n$-dimensional $3^{nk^{k}}$-horseshoe for $\phi_{0,\beta}$ and otherwise $\phi_{0,\beta}|_{E_{k}}$ is the identity on $E_{k}$.   We can prove    $      \underline{\text{mdim}}_{\text{M}}([0,1]^{n} ,\Vert \cdot\Vert,\phi_{0,\beta})=0$ and $  \overline{\text{mdim}}_{\text{M}}([0,1]^{n} ,\Vert \cdot\Vert,\phi_{0,\beta})=\beta.$

  {Now consider  $\alpha<\beta$}. From Lemma  \ref{lemma22}(i), we have there exists  $\phi_{\alpha}\in \text{Hom}([0,1]^{n})$  such that $$ {\underline{\text{mdim}}_{\text{M}}}([0,1]^{n} ,\Vert \cdot\Vert,\phi_{\alpha})={\overline{\text{mdim}}_{\text{M}}}([0,1] ^{n},\Vert \cdot\Vert,\phi_{\alpha})=\alpha.$$   Set $ {\phi}_{\alpha,\beta}\in \text{Hom}([0,1]^{n})$ be defined by \begin{equation*}   {\phi}_{\alpha,\beta} (x) =\begin{cases}
    T_{1}^{-1}\circ \phi_{0,\beta}\circ T_{1}(x) & \text{if }x\in [0,\frac{1}{2}]^{n},\\
     T_{2}^{-1}\circ \phi_{\alpha}\circ T_{2} (x)& \text{if }x\in [\frac{1}{2},1]^{n},\\
     x & \text{otherwise},
  \end{cases}\end{equation*}  where $T_{1}:[0,\frac{1}{2}]^{n}\rightarrow [0,1]^{n}$ and $T_{2}:[\frac{1}{2},1]^{n}\rightarrow [0,1]^{n}$ are bi-lipchitsz maps. 
 We have \begin{align*}  {\underline{\text{mdim}}_{\text{M}}}([0,1]^{n}& ,\Vert \cdot\Vert, {\phi}_{\alpha,\beta})\\
 &=\max\{{\underline{\text{mdim}}_{\text{M}}}([0,{1}/{2}]^{n} ,\Vert \cdot\Vert, {\phi}_{\alpha,\beta}|_{[0,\frac{1}{2}]^{n} }),{\underline{\text{mdim}}_{\text{M}}}([{1}/{2},1]^{n} ,\Vert \cdot\Vert, {\phi}_{\alpha,\beta}|_{[\frac{1}{2},1]^{n}})\}\\
 &=\max\{{\underline{\text{mdim}}_{\text{M}}}([0,1]^{n} ,\Vert \cdot\Vert, \phi_{0,\beta}),{\underline{\text{mdim}}_{\text{M}}}([0,1]^{n} ,\Vert \cdot\Vert,\phi_{\alpha})\}\\
 & = {\underline{\text{mdim}}_{\text{M}}}([0,1]^{n} ,\Vert \cdot\Vert,\phi_{\alpha})=\alpha .\end{align*}
 
 Analogously, we prove  $ {\overline{\text{mdim}}_{\text{M}}}([0,1]^{n} ,\Vert \cdot\Vert,\phi_{\alpha,\beta})={\overline{\text{mdim}}_{\text{M}}}([0,1]^{n} ,\Vert \cdot\Vert, \phi_{0,\beta})=\beta.$ 
 \end{proof}

\section{Homeomorphisms on manifolds  with positive metric mean dimension}\label{secvvnfnrk}

 Throughout  this section,  $N$ will   denote an $n$ dimensional   compact riemannian manifold with $n\geq 2$ and $d$ a riemannian metric on $N$.  On $\text{Hom}(N)$ we will consider the metric \begin{equation*}\label{cbenfn} \hat{d}(\phi,\varphi)=\max_{p\in N}\{d(\phi(p),\varphi(p)),d(\phi^{-1}(p),\varphi^{-1}(p))\}\quad \quad\text{ for any }\phi, \varphi \in   \text{Hom}(N).\end{equation*}
 It is well-known  $(\text{Hom}(N),\hat{d})$ is a complete metric space. 
  The main goal of this section is to prove  Theorem \ref{density_tipo1}. 
\medskip

\paragraph{\textbf{Good Charts}}
For each $p\in N$, consider the      \textit{exponential map}  $$\text{exp}_{p}: B_{\delta^{\prime}}(0_{p})\subseteq T_{p}N\rightarrow B_{\delta^{\prime}}(p)\subseteq N,$$ where $0_{p}$ is the origin in the tangent space $T_{p}N$,  $\delta^{\prime}$ is the \textit{injectivity radius} of $N$ and      $B_{\epsilon}(x)$ denote the open ball of radius $\epsilon>0$ with center $x$. 
 We will take $ {\delta}_{N}=\frac{\delta^{\prime}}{2}$.    
The exponential map has the following properties:
\begin{itemize}
\item Since $N$ is compact, $\delta^{\prime}$ does not depends on $p$.
\item $ \text{exp}_{p}(0_{p}) = p$ and $ \text{exp}_{p}[ B_{\delta_{N}}(0_{p})] = B_{\delta_{N}}({p})$;
    \item $\text{exp}_{p}: B_{\delta_{N}}(0_{p})\rightarrow B_{\delta_{N}}({p})$ is a diffeomorphism;
    \item If $v\in B_{\delta_{N}}(0_{p})$, taking $q=\text{exp}_{p}(v)$ we have $d(p,q)=\Vert v\Vert$. 
    \item The derivative of $\text{exp}_{p}$ at the origin is the identity map: $$D (\text{exp}_{p})(0) = \text{id} : T_{p}N \rightarrow  T_{p}N.$$
\end{itemize}

Since $\text{exp}_{p}: B_{\delta_{N}}(0_{p})\rightarrow B_{\delta_{N}}({p})$ is a diffeomorphism and $D (\text{exp}_{p})(0) = \text{id} : T_{p}N \rightarrow  T_{p}N,$ we have $\text{exp}_{p}: B_{\delta_{N}}(0_{p})\rightarrow B_{\delta_{N}}({p})$ is a bi-Lipschitz map with Lipschitz constant close to 1. Therefore, we can assume that if $v_{1},v_{2}\in B_{\delta_{N}}(0_{p})$, taking $q_{1}=\text{exp}_{p}(v_{1})$ and $q_{2}=\text{exp}_{p}(v_{2})$,  we have $d(q_{1},q_{2})=\Vert v_{1} -v_{2}\Vert$. Furthermore, we  will  identify  $B_{\delta_{N}}(0_{p})\subset T_{p}N$ with $B_{\delta_{N}}(0)=\{x\in\mathbb{R}^{n}:\Vert x\Vert <\delta_{N}\}\subseteq \mathbb{R}^{n}$.
\ \\

\noindent Recall that if $\alpha=\beta\in [0,n]$, then   $${H_{\alpha}^{\alpha}}(N): ={H}_{\alpha}(N)=\{\phi\in \text{Hom}(N): \underline{\text{mdim}}_\text{M}(N,  d,\phi)=\overline{\text{mdim}}_\text{M}(N,  d,\phi)=\alpha\}.$$

\begin{theorem}
\label{densitypositivemanifold}  For any $\alpha\in [0,n]$, the set $H_{\alpha}(N)  $ is   dense  in $\emph{Hom}(N)$.
\end{theorem}
\begin{proof} 
Let $P^{r}(N)$ be the set of $C^{r}$-diffeomorphisms on $N$ with a periodic point.  This set is $C^{0}$-dense in $\text{Hom}(N)$ (see \cite{Artin},    \cite{Hurley}).  Hence, in order to prove the theorem, it is sufficient to show that $H_{\alpha}(N)$ is dense in $P^r(N)$.   
Fix $\psi\in P^{r}(N)$
 and take any  $\varepsilon\in (0,\delta_{N})$. Suppose that $p\in N$ is a periodic point of $\psi$, with period $k$. Let $\beta>0$,  small enough, such  that $[-\beta, \beta]^n\subset  B_{\frac{\varepsilon}{4}}(0)$.  For each $i=0,\dots, k-1$, $N\setminus \psi(\exp_{\psi^{i}(p)}\partial[-\beta,\beta]^n)$  has two connected components. We denote by $C_{i}$ the connected component that contains $\psi^{i}(p)$.  Consider the positive number  
 \[\gamma=\min_{i=0,\dots, k-1}d(\psi^{i+1}(p), \psi(\exp_{\psi^{i}(p)}\partial [-\beta,\beta]^n))>0.\]
  Let $\lambda\in (0, \min\{\gamma/2, \beta\}),$ such that
   \[\exp_{\psi^{i+1}(p)} [-\lambda,\lambda]^n\subset C_{i+1}\quad\text{for each }i=0,\dots, k-1.\] The regions 
$[-\beta,\beta]^n\setminus (-\lambda, \lambda)^n $ and $ C_{i+1}\setminus \exp_{\psi^{i}(p)}(-\lambda, \lambda)^n $ are homeomorphic. For each $i=0,\dots, k-1, $ take \begin{align*}h_i: \partial [-\beta,\beta]^n&\to \exp^{-1}_{\psi^{i+1}(p)} \left(\partial  {C}_{i+1}\right)\\
u&\mapsto \exp^{-1}_{\psi^{i+1}(p)}\circ \psi\circ \exp _{\psi^{i}(p)}(u) .\end{align*} Then, there exists a homeomorphism $$H_i:   [-\beta,\beta]^n\setminus (-\lambda, \lambda)^n \to \exp^{-1}_{\psi^{i+1}(p)}[C_{{i+1}}\setminus \exp_{\psi^{i+1}(p)}(-\lambda, \lambda)^n],$$ such that 
\[H_{i}|_{\partial [-\lambda,\lambda]^n}=Id\ \ \text{and}\ \ H_{i}|_{\partial [-\beta,\beta]^n}=h_i.\]
  If $q\in R_i=\exp_{\psi^{i}(p)}  \left([-\beta,\beta]^n\setminus (-\lambda,\lambda)^n \right)$, we can write 

\[q=\exp_{\psi^{i}(p)}(z), \quad\text{for some } z\in [-\beta,\beta]^n\setminus (-\lambda,\lambda)^n. \]
Set $\varphi_{\alpha}: N\rightarrow N$, given by 
\[\varphi_{\alpha}(q)=\begin{cases}\exp_{\psi^{i+1}(p)}H_i(z), &\text{if }q\in R_i=\exp_{\psi^{i}(p)}  \left([-\beta,\beta]^n\setminus (-\lambda,\lambda)^n \right)\\
\exp_{\psi^{i+1}(p)}\phi_{\alpha}(\exp_{\psi^{i}(p)}^{-1}(q)), &\text{if }q\in D_i:=\exp_{\psi^{i}(p)}  \left([ -\lambda,\lambda]^n \right)\\
\psi(q),  & \text{otherwise}, \end{cases}\]
where  $\phi_{\alpha}\colon[ -\lambda,\lambda]^n \to [  -\lambda,\lambda]^n$ is a homeomorphism     which satisfies  $\phi_{\alpha}|_{\partial [ -\lambda,\lambda]^n }=Id$ and $$  {{\text{mdim}}_{\text{M}}}([ -\lambda,\lambda]^n  ,\Vert \cdot\Vert,\phi_{\alpha})=\alpha,$$ (see Lemma \ref{lemma22}). 
Set $K:=\cup_{i=0}^{k-1}D_{i}$. Note that $N\setminus K$ is $\varphi_{\alpha}$ invariant and     $$ \text{mdim}_\text{M}(N\setminus K ,d,\varphi_{\alpha}|_{N\setminus K})=0.$$
 Hence,   $$ \text{mdim}_\text{M}(N ,d,\varphi_{\alpha})=  \text{mdim}_\text{M}(K  ,d,\varphi_{\alpha}|_{K}).$$ 
   Note if $q\in D_{i}$, we have
$$  (\varphi_{\alpha})^{s}(q) =   \text{exp}_{\psi^{(i+s)\text{ mod }k}(p)}\circ(\phi_{\alpha})^{s}\circ\text{exp}^{-1}_{\psi^{i}(p)}(q) \quad\text{and}\quad   (\varphi_{\alpha})^{k}(q) =   \text{exp}_{\psi^{i}(p)}\circ(\phi_{\alpha})^{k}\circ\text{exp}^{-1}_{\psi^{i}(p)}(q) .$$
Hence,  $D\subseteq D_{i}$   is {an} $(s,\phi_{\alpha},\epsilon)$-{separated} set if and only if $\text{exp}_{\psi^{i}(p)}(D)\subseteq N$   is {an} $(s,\varphi_{\alpha},\epsilon)$-{separated} set for any $\epsilon>0$. Therefore,   $\text{sep}(s,\varphi_{\alpha}|_{K},\epsilon)=k\,\text{sep}(s,\phi_{\alpha} ,\epsilon).$ Therefore,  $$  \text{mdim}_\text{M}( K ,d,\varphi_{\alpha}|_{K})= \text{mdim}_\text{M}([-\lambda,\lambda]^{n} ,\Vert \cdot\Vert,\phi_{\alpha}),$$   
 which proves the theorem. 
\end{proof}
The last theorem proved Theorem \ref{density_tipo1} in the case $\alpha=\beta =n$. The proof of the general case will be a consequence of the above arguments, in fact:

  \begin{proof}[\emph{\textbf{Proof of the Theorem \ref{density_tipo1}}}] The proof  follows similarly to the proof of Theorem \ref{densitypositivemanifold}, changing $\phi_{\alpha}$  by a continuous map  $\phi_{\alpha,\beta}\colon [-\lambda, \lambda]^n \to [-\lambda, \lambda]^n$  such that 
  \begin{equation*}{\underline{\text{mdim}}_{\text{M}}}([-\lambda, \lambda]^{n} ,\Vert \cdot\Vert, {\phi}_{\alpha,\beta})=\alpha\quad\text{and}\quad {\overline{\text{mdim}}_{\text{M}}}([-\lambda, \lambda]^{n} ,\Vert \cdot\Vert,\phi_{\alpha,\beta})=\beta,\end{equation*}
   as in the Lemma \ref{med1}. 
 \end{proof}

 \section{Tipical homeomorphism has maximal metric mean dimension}\label{Section6}
  

 To complete this work, in this section we show  Theorem \ref{teoresidual*}, which was proved in Theorem A in \cite{Carvalho}, however, we will present an alternative proof of this fact using the techniques of Section \ref{section55} and Section \ref{secvvnfnrk}.

 \begin{definition}[$n$-dimensional    strong horseshoe]\label{3ddimensionalhorse} Let $E=[a,b]^{n}$ and set $|E|=b-a$. For a fixed natural number $k>1$, set  $\delta_{k}=  \frac{b-a}{4k+1}$ and $\epsilon_{k}=\frac{b-a}{2(2k+1)^{n-1}-1}$.  For $i=0,1, 2, \dots, 4k+1$, set $t_i= a+i\delta_{k} $ and consider   $$H_{i_{1},i_{2},\dots,i_{n-1}}=[a,b]\times[t_{i_{1}-1},t_{i_{1}}]\times  \cdots \times[t_{i_{n-1}-1}, t_{i_{n-1}}],$$ for $i_{j}\in\{1,\dots, 4k+1\}.$  For $l=0,1, 2, \dots, 2(2k+1)^{n-1}-2$, set $s_l= a+l\epsilon_{k} $ and consider   $$ V_{l+1}=[s_l, s_{l+1}]\times [a,b]^{n-1}.$$ 
  We say that $E\subseteq A\subseteq \mathbb{R}^{n}$ is an     $n$-\textit{dimensional   strong   $(\epsilon,(2k+1)^{n-1})$-horseshoe} for a homeomorphism $\phi:A\rightarrow A$ if $|E|>\epsilon$ and furthermore:  
\begin{itemize} \item For  any $H_{i_{1},i_{2},\dots,i_{n-1}}$, with $i_{j}\in\{1,3,\dots, 4k+1\}$, there exists     $l\in\{1,3,\dots, 2(2k+1)^{n-1}-1\}$ with $$H_{i_{1},i_{2},\dots,i_{n-1}}\subseteq \phi(V_{l})^{\circ}.$$ 
\item For any $l=2,4,\dots, 2(2k+1)^{n-1}-2$,  $\phi(V_{l})\subseteq A\setminus E$.\end{itemize} 
  \end{definition}
 
  \begin{remark}
  Note if $\epsilon^{\prime}>\epsilon$, then any $n$-dimensional strong $(\epsilon^{\prime},(2k+1)^{n-1})$-horseshoe for $ \phi$ is an $n$-dimensional strong $(\epsilon,(2k+1)^{n-1})$-horseshoe for $ \phi$.
  \end{remark}
  
Using local charts, the last   definition can be done on the manifold $N$.

\begin{definition} Let $N$ be an $n$-dimensional riemannian manifold and fix $k\geq 1$. We say that $\phi\in \text{Hom}(N)$ has an  $n$-\textit{dimensional strong} $( \epsilon, (2k+1)^{n-1})$-\textit{horseshoe} $E=[a,b]^{n}$,  if there is $s$ and an exponential charts $\text{exp}_{i}: B_{\delta_{N}}(0)\rightarrow N$, for $i=1,\dots,s$,   such that: \begin{itemize} \item $\phi_{i}=\text{exp}_{(i+1)\text{mod}\,  s}\circ \phi \circ \text{exp}^{-1}_{i}: B_{\delta}(0 )\rightarrow B_{\delta_{N}}(0)$ is well defined for some $\delta \leq \delta_{N}$;
 \item $E$ is an $n$-dimensional  strong $(\epsilon, (2k+1)^{n-1})$-horseshoe for $\phi_{i} $. \end{itemize}  
 To simplify the notation, we will set $\phi_{i}=\phi$ for each $i=1,\dots, s$. \end{definition}
 
\noindent  For $\epsilon >0$ and $k\in\mathbb{N}$, we consider the sets 
\begin{align*} &H(\epsilon, k)=\{\phi^{2} \in \text{Hom}(N):  \phi \text{ has a strong }(\epsilon,k) \text{-horseshoe} \}  \\
&H(k)=\bigcup_{i\in\mathbb{N}}H\left( \frac{1}{ i^{2}},3^{ki(n-1)} \right)\\
 & \mathcal{H}=\overset{\infty }{\underset{k=1}{\bigcap}} H(k). \end{align*}  
 \begin{lemma}
The set $\mathcal{H}$ is residual. 
 \end{lemma} 
 \begin{proof}
  Clearly, for any $\epsilon\in (0,\delta_{N})$ and $k\in\mathbb{N}$, the set  $ H(\epsilon, k)$  is open and nonempty.  \\
We claim that the set $H(k)$ is dense in $\text{Hom}(N)$. In fact: fix  $\psi\in P^{r}(N)$ with a $s$-periodic point. In the same way as the proof of the Theorem \ref{densitypositivemanifold},  every small  neighborhood of the orbit of this point can be perturbed in order to obtain a strong $\left(\frac{1}{i^{2}},3^{n k i}\right)$   horseshoe for   a $\phi$ such that $\phi^{2}$ be close to $\psi$ for a large enough $i$. Thus $H(k)$ is a dense set, and  \emph{a fortiori}   $\mathcal{H}=\overset{\infty }{\underset{k=1}{\bigcap}} H(k) $  is residual in $\text{Hom}(N)$.
\end{proof}
Finally, we prove  Theorem \ref{teoresidual*}.
\begin{proof}[\emph{\textbf{Proof of the Theorem \ref{teoresidual*}}}]

It is sufficient to prove that $\overline{\text{mdim}}_{\text{M}}(N,d,\phi)=n$ for any $\phi\in \mathcal{H}$.  For this sake, take $\varphi=\phi^{2}\in \mathcal{H}$. We have $\varphi \in H(k)$ for any $k\geq 1$. Therefore, for any $k\in\mathbb{N}$, there exists   $i_{k}$, with $i_{k}<i_{k+1}$,  such that $\phi $ has a strong  $\left( \frac{1}{i_{k}^{2}},3^{n\,k\, i_{k}} \right)$-horseshoe $E_{{k}}=[a_{k},b_{k}]^{n}$, such that $|E_{{k}}|>\frac{1}{i_{k}^{2}}$, consisting of   rectangles  $$H^{i_{k}}_{j_{1},j_{2},\dots,j_{n-1}}=[a_{k},b_{k}]\times[t^{{k}}_{j_{1}-1},t^{{k}}_{j_{1}}]\times  \cdots \times[t^{{k}}_{j_{n-1}-1}, t^{{k}}_{j_{n-1}}], \quad\text{with}\quad j_{t}\in\{1,3,\dots, 2(3^{k \,i_{k}})-1\},$$ and $$V^{{k}}_{l}=[s^{k}_{l-1}, s^{k}_{l}]\times[a_{k},b_{k}]^{n-1},\quad\text{    for  }l\in\{1,3,\dots, 2( {3^{k i_{k}(n-1)}})-1\},$$  with $$H^{{k}}_{j_{1},j_{2},\dots,j_{n-1}}\subseteq \phi(V^{{k}}_{l})^{\circ}\quad\text{for some } l.$$ 
   For $i_{j}\in\{1,\dots,  2(3^{ki_{k}})-1\}$ and $l=1,\dots, 2(3^{ki_{k}})^{n-1}-1,$ 
 take $$E^{k}_{l,i_{1},\dots,i_{n-1}}   =[s^{k}_{l-1},s^{k}_{l}]\times[t^{k}_{i_{1}-1},t^{k}_{i_{1}}]\times  \cdots \times[t^{k}_{i_{n-1}-1}, t^{k}_{i_{n-1}}] .$$  
For $t=1,\dots,m$, let  $i_{1}^{(t)},$ $ \dots,i_{n-1}^{(t)}\in\{1,3,\dots, 2(3^{ki_{k}})-1\}$  and for $l_{t}\in\{1,\dots, 2(3^{ki_{k}})^{n-1}-1\},$        set   \begin{align*} C^{k}_{l_{1}, i_{1}^{(1)},\dots, i_{n-1}^{(1)}}& =   E^{k}_{l_{1},i_{1}^{(1)},\dots,i_{n-1}^{(1)}}\\
 C^{k}_{l_{2},i_{1}^{(2)},\dots, i_{n-1}^{(2)},l_{1},i_{1}^{(1)},\dots, i_{n-1}^{(1)}} &=\varphi^{-1}\left[\varphi\left(C^{k}_{l_{2}, i_{1}^{2},\dots, i_{n-1}^{2}}\right)\cap E^{k}_{l_{1},i_{1}^{(1)},\dots,i_{n-1}^{(1)}}\right]\\
  &\vdots\\
 C^{k}_{l_{m},i_{1}^{(m)},\dots,i_{n-1}^{(m)},  \dots,l_{1}, i_{1}^{(1)},\dots, i_{n-1}^{(1)}} &=\varphi^{-(m-1)}\left[\varphi^{m-1}\left(C^{k}_{l_{m},i_{1}^{(m)},\dots,i_{n-1}^{(m)}, \dots,l_{2}, i_{1}^{(2)},\dots,i_{n-1}^{(2)}}\right)\cap E^{k}_{l_{1},i_{1}^{(1)},\dots,i_{n-1}^{(1)}}\right]
 \end{align*}  Furthermore,  set $$ \tilde{C}^{k}_{l_{m},i_{1}^{(m)},\dots,i_{n-1}^{(m)},  \dots,l_{1}, i_{1}^{(1)},\dots, i_{n-1}^{(1)}} :=\text{exp}\left[ C^{k}_{l_{m},i_{1}^{(m)},\dots,i_{n-1}^{(m)},  \dots,l_{1}, i_{1}^{(1)},\dots, i_{n-1}^{(1)}} \right].$$ 
For any $k\geq 1$,   take  $\varepsilon_{k}=\frac{1}{4i_{k}^{2}(2(3^{k \,i_{k} })-1)}$. We can choose at least $3^{ki_{k}}$ sub-indices $p_{1},\dots,p_{3^{ki_{k}}}$ in $\{1,\dots,  2(3^{ki_{k}})^{n-1}-1\}$ such that if $x\in \tilde{C}_{p_{s}, i_{1},\dots, i_{n-1}}$ and $y\in \tilde{C}_{p_{t}, i_{1},\dots, i_{n-1}}$, with $s\neq t$, then $d_{1}(x,y)>\varepsilon_{k}$. 
    Hence, if $x $ and $y$ belong to different sets  $  \tilde{C}^{k}_{p^{(m)},i_{1}^{(m)},\dots,i_{n-1}^{(m)},  \dots,p^{(1)}, i_{1}^{(1)},\dots, i_{n-1}^{(1)}} ,$ with    $p^{(t)}\in\{p_{1},\dots,p_{3^{ki_{k}}}\}$ and  $i_{1}^{(t)},$ $ \dots,i_{n-1}^{(t)}\in\{1,3,\dots, 2(3^{ki_{k}})-1\}$, we have that $d_{m}(x,y)>\varepsilon_{k}$.   Hence,
 \begin{equation*}    \text{sep}(m, \varphi  , \varepsilon_{k}) \geq     3^{ki_{k}nm} \quad \text{and thus}\quad 
  \text{sep}(\varphi   , \varepsilon_{k})\geq    \log  {3^{nki_{k}}}.
 \end{equation*}
 Therefore,  \begin{align*}\label{exxample12} {\overline{\text{mdim}}_{\text{M}}}(N ,d,\varphi  )& = \limsup_{k\rightarrow \infty} \frac{\text{sep}(\varphi  ,\varepsilon_{k} )}{|\log  {\varepsilon_{k}}|} \geq \lim_{k\rightarrow \infty}\frac{ \log 3^{nki_{k}}}{|\log \varepsilon_{k}|} =\limsup_{k\rightarrow \infty}\frac{\log 3^{nki_{k}}}{\log [4i_{k}^{2}(2(3^{ki_{k}})-1)]}  =n.\end{align*}   
  This fact   proves the theorem, since for any $\psi \in \text{Hom}(N)$, the inequality $   \overline{\text{mdim}}_{\text{M}}(N ,d ,\phi)\leq n$  always holds (see \cite{VV}, Remark 4).
\end{proof}
 
\section{Declarations}

\noindent \textbf{Ethical Approval}

\noindent This declaration is not applicable.
 
 \medskip
\noindent \textbf{Competing interests}  

\noindent We have no conflict of interest.

 \medskip
\noindent \textbf{Authors' contributions}  

\noindent All authors contributed equally to writing the manuscript.

 \medskip
\noindent \textbf{Funding} 

\noindent We were not financed.

 \medskip
\noindent \textbf{Availability of data and materials}

\noindent This declaration is not applicable.

 \end{document}